\documentclass[11pt]{article}
\usepackage[utf8]{inputenc}
\usepackage{cite}

\title{{\Large \vspace{-1cm} The mapping class group action on the odd character variety is faithful}}
\author{Aliakbar Daemi and Christopher Scaduto} 

\date{}

\usepackage[a4paper, total={6in, 8.5in}]{geometry}
\usepackage{graphicx}

\usepackage{times}

\usepackage{amssymb}
\usepackage{tikz-cd}
\usepackage{titling}
\usepackage{mathrsfs} 
\usepackage{booktabs}
\usepackage[all]{xy}
\usepackage{amsthm}
\usepackage{diagbox}
\usepackage{tabularx}
\usepackage{amscd}
\usepackage{caption}
\usepackage{nicefrac}
\usepackage{amsmath}
\usepackage{mathtools}
\usepackage{tikz}
\usepackage{mathabx}
\usepackage{dsfont}
\usepackage{lipsum}
\usepackage{mwe}
\usepackage{slashed}
\usepackage{rotating}
\usepackage{subcaption}
\usepackage[colorlinks,pagebackref,hypertexnames=false]{hyperref} \usepackage[alphabetic,backrefs,msc-links]{amsrefs}
\usepackage{epstopdf,pinlabel}
\definecolor{mint}{HTML}{239B56}
\hypersetup{
    colorlinks=true,
    linkcolor=blue,
    filecolor=magenta,      
    urlcolor=cyan,
    citecolor=mint
}
\usepackage{pgfplots}
\pgfplotsset{compat=newest}
\usetikzlibrary{shapes.geometric}

\definecolor{greenish}{rgb}{0.01, 0.75, 0.24}
\definecolor{blueish}{rgb}{0.0, 0.72, 0.92}
\definecolor{orangeish}{rgb}{1.0, 0.55, 0.0}

\newcolumntype{Y}{>{\centering\arraybackslash}X}

\usepackage{sectsty}

\sectionfont{\fontsize{12}{15}\selectfont}

\newcommand{\Z}{\mathbb{Z}}

\newcommand{\Q}{\mathbb{Q}}

\newtheorem{theorem}{Theorem}[section]
\newtheorem{prop}[theorem]{Proposition}
\newtheorem{lemma}[theorem]{Lemma}

\newtheorem{conjecture}[theorem]{Conjecture}

\newtheorem{problem}[theorem]{Problem}

\newtheorem{remark}[theorem]{Remark}

\newcommand{\Addresses}{{
 \bigskip
 \footnotesize
 Aliakbar Daemi, \textsc{Department of Mathematics, Washington University in St. Louis, One Brookings drive, Room 207A,
 St. Louis, MO 63130}\par\nopagebreak
 \textit{E-mail address}: \texttt{adaemi@wustl.edu}
 \vspace{.4cm}

Christopher Scaduto, \textsc{Department of Mathematics, University of Miami, 1365 Memorial Dr 515, Coral Gables, FL 33124}\par\nopagebreak
 \textit{E-mail address}: \texttt{cscaduto@miami.edu}
}}

\begin{document}

\maketitle

\vspace{-0.75cm}

\begin{abstract}{The odd character variety of a Riemann surface is a moduli space of $SO(3)$ representations of the fundamental group which can be interpreted as the moduli space of stable holomorphic rank $2$ bundles of odd degree and fixed determinant. This is a symplectic manifold, and there is a homomorphism from a finite extension of the mapping class group of the surface to the symplectic mapping class group of this moduli space. When the genus is at least $2$, it is shown that this homomomorphism is injective. This answers a question posed by Dostoglou and Salamon and generalizes a theorem of Smith from the genus $2$ case to arbitrary genus. A corresponding result on the faithfulness of the action on the Fukaya category of the odd character variety is also proved. The proofs use instanton Floer homology, a version of the Atiyah--Floer Conjecture, and aspects of a strategy used by Clarkson in the Heegaard Floer setting.}
\end{abstract}

\vspace{.2cm}

\section{Introduction}

Let $\Sigma_g$ be a Riemann surface of genus $g$. Consider the moduli space $M_g$ of flat connections on a non-trivial $SO(3)$-bundle $P$ over $\Sigma_g$ modulo the gauge group consisting of gauge transformations that lift to $SU(2)$. Via holonomy, $M_g$ may be identified with the space of conjugacy classes of homomorphisms from $\pi_1(\Sigma_g\setminus p)$ to $SU(2)$ which map loops around $p\in \Sigma_g$ to $-1$. Because of this description, $M_g$ is sometimes referred to as the odd character variety. By a result of Narasimhan and Seshadri \cite{narasimhan-seshadri}, $M_g$ may be identified with the moduli space of stable holomorphic rank $2$ bundles over $\Sigma_g$ with some fixed odd degree determinant. The moduli space $M_g$ is a smooth closed manifold of dimension $6g-6$ and has a natural symplectic structure \cite{atiyah-bott}. It has been extensively studied for more than half a century, see for example \cite{thaddeus-intro}.

Any diffeomorphism of $\Sigma_g$ lifts to an automorphism of the bundle $P$ and the space of all such lifts up to the action of the gauge group is a copy of $H^1(\Sigma_g;\Z/2)$. This determines an 
extension of the mapping class group $\text{Mod}(\Sigma_g)=\pi_0 \text{Diff}^+(\Sigma_g)$ of the form
\begin{equation}\label{eq:theexactsequence}
    1\to H^1(\Sigma_g;\Z/2) \to \widehat{\Gamma}_g \to \text{Mod}(\Sigma_g) \to 1.
\end{equation}
Taking pullback of flat connections on $P$ with respect to a representative of an element of $\widehat{\Gamma}_g$ determines a symplectomorphism of $M_g$. There is an induced homomomorphism
\begin{equation*}\label{eq:rhohat}
    \widehat{\rho}:\widehat{\Gamma}_g \to \pi_0 \text{Symp}(M_g)
\end{equation*}
where the target is the symplectic mapping class group of $M_g$. The main result proved here is: 

\vspace{0.2cm}

\begin{theorem}\label{thm:main}
For $g\geq 2$, the homomorphism $\widehat{\rho}$ is injective.
\end{theorem}

\vspace{0.2cm}

\noindent The question of whether $\widehat{\rho}$ is injective was raised in the early 1990s by Dostoglou and Salamon \cite[Remark 5.6]{dostoglou-salamon}, see also \cite{seidel-thesis}. The homomorphism $\widehat{\rho}$ was studied by Smith \cite{smith}, who showed that $\widehat{\rho}$ is injective if $g=2$; he furthermore showed, for all $g$, that $\widehat{\rho}$ does not factor through $\text{Sp}(2g,\Z)$. One may also consider the homomorphism $\widehat{\rho}$ composed with the forgetful map to $\pi_0 \text{Diff}^+(M_g)$. However, there is evidence that this homorphism has large non-trivial kernel in general, see the introductory discussions in \cite{smith,seidel-thesis}.

In \cite{smith}, Smith establishes the injectivity of $\widehat{\rho}$ for $g=2$ by first proving an equivalence between a version of the Fukaya category of the surface $\Sigma_2$ and that of the moduli space $M_2$, which may be viewed as a version of Witten's Conjecture \cite{witten-monopoles}, see also \cite[\S 1.6]{smith}. Smith then obtains the $g=2$ version of Theorem \ref{thm:main} from known results about the action of the mapping class group on Fukaya categories of Riemann surfaces. 

The methods of this paper are different than those of \cite{smith} and rely on 3-manifold topology. Nevertheless, we obtain the following result about the Fukaya category of the moduli space $M_g$ which has Theorem \ref{thm:main}, in the cases $g\geq 3$, as an immediate consequence. 

\begin{theorem}\label{thm:main-2}
    For $g\geq 3$, the action of $\widehat{\rho}$ on the monotone Fukaya category of $M_g$ is faithful. Specifically, for $\phi\in \smash{\widehat{\Gamma}_g}\setminus \{0\}$, there is a pair of simply-connected Lagrangians $L$ and 
    $L'$ in $M_g$ such that 
    the Lagrangian Floer homology groups ${\rm HF}(L,L')$ and 
    ${\rm HF}(L,\widehat{\rho}(\phi)(L'))$ are not isomorphic. 
\end{theorem}

\noindent The Lagrangian Floer groups appearing are defined as in \cite{oh}, built upon the work of Floer \cite{floer-lagrangian}. Note that $M_g$ is monotone and thus any simply-connected Lagrangian in $M_g$ is also monotone in the sense of \cite{oh}. For the case $g=2$, our arguments show that the action of $\widehat{\rho}$ on the Fukaya category has kernel contained in an order $2$ subgroup of $\widehat{\Gamma}_2$ generated by a hyperelliptic involution. See Remark \ref{rmk:genus2} for more details.

The proofs of Theorems \ref{thm:main} and \ref{thm:main-2} involve several ingredients. The first is Floer's $SO(3)$ instanton homology for admissible bundles over $3$-manifolds \cite{floer-dehn} and a non-vanishing criterion in this setting for irreducible $3$-manifolds due to Kronheimer and Mrowka (Proposition \ref{prop:nonvanishinginstanton}). The second is a version of the Atiyah--Floer Conjecture for admissible bundles (Theorem \ref{thm:atiyahfloer}) proved by Fukaya, Lipyanskiy and the first author \cite{dfl}; this relates instanton homology to the symplectic topology of $M_g$. Finally, several arguments from Clarkson's work in the Heegaard Floer setting \cite{clarkson} are adapted and used, together with the previous ingredients, to obtain the results.

The idea of using instanton homology to understand the symplectic mapping class group of $M_g$ is implicit in the work of Dostoglou and Salamon \cite{dostoglou-salamon}, who proved a version of the Atiyah--Floer Conjecture for mapping tori.

This paper is organized as follows. In Section \ref{sec:background}, the requisite background on instanton Floer homology is reviewed, and in Section \ref{sec:proof} the main line of argument for Theorems \ref{thm:main} and \ref{thm:main-2} is carried out. In Section \ref{sec:further}, additional questions and conjectures are discussed.\\ 

\noindent \textbf{Acknowledgments.} A related  question about the action of the mapping class group of Riemann surfaces on $SU(2)$ character varieties was proposed by Robert Lipshitz for the K3 problem list (see Section \ref{sec:further}), where he suggests that Clarkson's work \cite{clarkson} may be useful in addressing the problem. The authors are grateful to him for drawing their attention to \cite{clarkson}. They also thank Ivan Smith for an informative correspondence and Ken Baker and Bill Goldman for helpful discussions. The first author was supported by NSF Grant DMS-2208181 and NSF FRG Grant DMS-1952762 and the second author was supported by NSF Grant DMS-1952762.

\section{Background on instanton Floer homology}\label{sec:background}

Let $(Y,w)$ be a pair consisting of a connected closed oriented $3$-manifold $Y$ and an embedded unoriented $1$-manifold $w\subset Y$. There is an $SO(3)$-bundle $E_w\to Y$ naturally associated to $w\subset Y$ such that $[w]\in H_1(Y;\Z/2)$ is Poincar\'{e} dual to $w_2(E_w)$. An embedded orientable surface $R\subset Y$ is {\emph{nice}} if the algebraic intersection $w\cdot R$ is odd. The pair $(Y,w)$ is {\emph{admissible}} if it admits a nice surface. Given such a pair, Floer constructed \cite{floer-dehn} its instanton homology 
\[
    I_\ast(Y,w),
\]
an abelian group with a relative $\Z/8$-grading.  This group is determined up to isomorphism by $Y$ and the class $[w]\in H_1(Y;\Z/2)$. As explained in \cite{km-sutures}, a nice surface for $(Y,w)$ induces a $4$-periodic involution on $I_\ast(Y,w)$, and the quotient $I_\ast(Y,w)_R$ is a relatively $\Z/4$-graded abelian group.

\vspace{0.15cm}

\begin{prop}\label{prop:nonvanishinginstanton}
     Suppose $(Y,w)$ is an admissible pair. Then we have:
     \begin{equation}\label{eq:2spherevanishing}
I_\ast(Y,w) =0 \quad \Longleftrightarrow\quad  Y\cong Y'\# S^1\times S^2,\;\;
    w\cdot S^2 \text{ odd}.
\end{equation}
In particular, if $Y$ is irreducible, then $I_\ast(Y,w)\neq 0$.
\end{prop}

\begin{proof}
The ``$\Leftarrow$'' implication follows easily from the construction of $I_\ast(Y,w)$ and the fact that there are no flat connections on the non-trivial $SO(3)$-bundle over the $2$-sphere. To prove the ``$\Rightarrow$'' direction, we use the framed instanton homology 
\[
    I^\#(Y,w):= I_\ast(Y\# T^3,w\cup u)_T
\]
where $u:=S^1\times \{\text{pt}\}$ and $T:= \{\text{pt}\}\times S^1\times S^1$. Note that $I^\#(Y,w)$ is defined for any pair $(Y,w)$, not just admissible ones. The argument of \cite[Cor. 5.21]{dis}, relying on machinery of \cite{km-sutures}, shows that $I^\#(Y,w; \Q) \neq 0$ if $Y$ is irreducible and $b_1(Y)>0$. Indeed, these conditions guarantee the existence of a surface with non-trivial homology class and positive Thurston norm. If $b_1(Y)=0$, by \cite[Cor. 1.4]{scaduto-thesis} the Euler characteristic of the framed instanton homology of $(Y,w)$ for any $w\subset Y$ is equal to $|H_1(Y;\Z)|$ and thus $I^\#(Y,w;\Q)\neq 0$. Using the K\"{u}nneth formula of \cite[\S 5.5]{km-unknot}, 
\[
    I^\#(Y\# Y', w\cup w';\Q) \cong I^\#(Y,w;\Q)\otimes I^\#(Y',w';\Q),
\]
together with the prime decomposition of $Y$ and the computation $I_\ast^\#(S^1\times S^2,\emptyset)=\Z^2$, we obtain that $I^\#(Y,w;\Q)  \neq  0$ if and only if $Y$ splits off an $S^1\times S^2$ summand where $w$ has odd pairing with $S^2$. Finally, for any admissible pair $(Y,w)$, the connected sum result of \cite[\S 9.7]{scaduto-thesis} shows that the rank of $I_\ast(Y,w)$ is bounded below by that of $I^\#_\ast(Y,w)$, yielding the claim.
\end{proof}

\vspace{0.15cm}

\begin{remark}
    {\normalfont The result that $I_\ast(Y,w)$ is non-zero for admissible pairs $(Y,w)$ with $Y$ irreducible is originally due to Kronheimer and Mrowka \cite{km-witten,km-icm}, and this is the substantive part of the proposition that is used below. Tye Lidman was also aware of an extension of Kronheimer and Mrowka's result to more general admissible pairs. 
    }
\end{remark}

\vspace{0.15cm}

An {\emph{admissible splitting}} of an admissible pair $(Y,w)$ is a decomposition 
\begin{equation}\label{eq:admissibledecomposition}
    (Y,w)= (Y_1,w_1)\cup_{(\Sigma,v)} (Y_2,w_2)
\end{equation}
where each connected component of $\Sigma$ is a nice surface for $(Y,w)$, $Y_1$ and $Y_2$ are connected oriented 3-manifolds with $\partial Y_1 = \Sigma = -\partial Y_2$, and $w_i\subset Y_i$ are properly embedded unoriented $1$-manifolds with $\partial w_i=v$, where $v$ is a collection of points on $\Sigma$. Note that $\# v$ is necessarily an even integer, while the number of points in the intersection of $v$ and each connected component of $\Sigma$ is odd. In particular, $\Sigma$ has an even number of connected components, and as an additional requirement on admissible splittings, we require that $\Sigma$ has two connected components. The bundle $E_w$ over $Y$ restricts to bundles $E_{w_i}$ over $Y_i$ and a bundle $E_v$ over $\Sigma$. The moduli  space $M(\Sigma,v)$ of flat connections on $E_v$ modulo the determinant 1 gauge group is the smooth closed symplectic manifold given as the product of the odd character varieties associated to the components of $\Sigma$. The flat connections on $E_{w_i}$, after a small perturbation, determine an immersed Lagrangian $L(Y_i,w_i)\subset M(\Sigma,v)$. Fukaya, Lipyanskiy and the first author proved the following result, which is a variation on what is known as the Atiyah--Floer Conjecture \cite{atiyah-newinvs}.

\vspace{0.15cm}

\begin{theorem}[\cite{dfl}]\label{thm:atiyahfloer} Suppose that there exist small perturbations such that the Lagrangians $L_i:=L(Y_i,w_i)\subset M(\Sigma,v)$ are embedded. Then the Lagrangian Floer homology ${\rm{HF}}_\ast(L_1,L_2)$ is defined as a relatively $\Z/4$-graded abelian group, and there is a homogeneously graded isomorphism:
\[
    I_\ast(Y,w)_\Sigma \cong {\rm{HF}}_\ast(L_1,L_2).
\]
\end{theorem}
\vspace{0.25cm}

We consider a particular family of decompositions \eqref{eq:admissibledecomposition}. Let $H$ be a standard oriented handlebody of genus $g\geq 2$, obtained from the 3-ball by attaching $g$-many $1$-handles. Choose an identification of $\partial H$ with $\Sigma_g$, a closed oriented surface of genus $g$. Let $c\subset H$ be a loop in $H$ formed by taking the core of one of the $1$-handles and closing it up inside the $3$-ball. Remove an open regular neighborhood $N(c)\subset \text{int}(H)$ of $c$ to obtain $H^{\circ}$. Let $T\subset \partial H^\circ$ denote the boundary $2$-torus of the closure of $N(c)$. Note that $\partial H^\circ$ is the disjoint union of $T$ and $\Sigma_g$. The manifold $H^\circ$ is an example of a compression body, and may also be viewed as obtained from $[-1,1]\times T^2$ by attaching $(g-1)$-many $1$-handles to $\{1\}\times T^2$, with $T$ corresponding to $\{-1\}\times  T^2$.

\begin{figure}[t]
\centering
\begin{tikzpicture}
    \node at (0,0) {\includegraphics[scale=2.4]{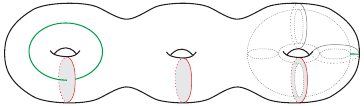}};
    \draw[gray] (-6.65,-2) -- (-5.98,-0.6);
    \node[below right] at (-7,-2) {$w'$};
    \draw[gray] (7.2,-2) -- (7,-0.07);
    \node[below right] at (7,-2) {$w''$};
    \draw[gray] (2.45,-1.95) -- (3,-1);
    \node[below right] at (2,-2) {$T$};
    \draw[gray] (-2.35,-1.95) -- (-2.7,-1.5);
    \node[below right] at (-2.4,-2) {$\Sigma_g$};
\end{tikzpicture}
\caption{\small The compression body $H^\circ$ is a genus $g$ handlebody (here $g=3$) with an open neighborhood of a core $1$-handle removed. The boundary created by this removal is the $2$-torus $T$, which in the picture sits interior to the outer boundary $\Sigma_g$. The circle $w'$ is another core, and $w''$ is an arc from $T$ to $\Sigma_g$.}
\label{fig:compressionbody}
\end{figure}
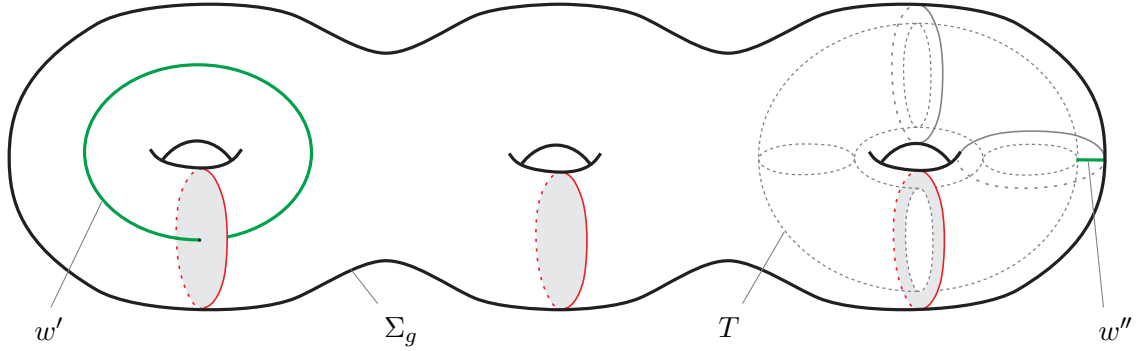

Let $w'\subset \text{int}(H^\circ)$ be a circle formed by the core of another $1$-handle, and $w''\subset H^\circ$ a small properly embedded arc that runs from $t\in T$ to $p\in\Sigma_g$. See Figure \ref{fig:compressionbody}. Let $\phi:\Sigma_g\to \Sigma_g$ be an orientation-preserving diffeomorphism such that $\phi(p)=p$. Take two copies $H^\circ_1$ and $H^\circ_2$ of $H^\circ$, where $H^\circ_1$ is oriented as is $H^\circ$, and $H^\circ_2$ oriented oppositely. Define $w_1=w_1'\cup w_2''\subset H^\circ_1$ and $w_2=w_2''\subset H^\circ_2$. Then we may form an admissible pair
\begin{equation}\label{eq:ouradmissibledecomposition}
(Y_\phi, w_\phi) := (H^\circ_1, w_1) \cup_{(\Sigma_g\cup T, \{p,t\}) } (H_2^\circ, w_2)
\end{equation}
by gluing $\partial H^\circ_1 \to \partial H_2^\circ$ using $\phi$ on the $\Sigma_g$ boundary components, and using the identity map on the $T$ boundary components. By construction, \eqref{eq:ouradmissibledecomposition} is an admissible decomposition of $(Y_\phi,w_\phi)$.

Noting that the genus $1$ moduli space $M_1=M(T,t)$ consists of a single point, the relevant symplectic manifold $M(\Sigma_g\cup T, \{p,t\})$ is naturally identified with $M_g$.

\vspace{0.15cm}

\begin{lemma}
For $i\in \{1,2\}$, the unperturbed Lagrangian $L_i:=L(H_i^\circ,w_i)$ is embedded in $M_g$.
\end{lemma}

\begin{proof}
This follows from \cite[Prop. 2.21]{dfl}, which in the general case of a decomposition \eqref{eq:admissibledecomposition} states that the unperturbed Lagrangian $L(Y_1,w_1)$ is embedded in $M(\Sigma,v)$ if for some connected component $\Sigma_0\subset \Sigma$ the inclusion induces a surjection $\pi_1(\Sigma_0)\to \pi_1(Y_1)$. In the case of interest, the inclusion of the boundary $\Sigma_g\subset \partial H^\circ$ clearly induces a surjection $\pi_1(\Sigma_g)\to \pi_1(H^\circ)$.
\end{proof}

\vspace{0.25cm}

The construction of $(Y_\phi,w_\phi)$ depends only on the class $[\phi]\in \text{Mod}(\Sigma_g,p):=\pi_0 \text{Diff}^+(\Sigma_g,p)$. Recall the Birman exact sequence (see \cite[\S 4.2]{farb-margalit}):
\begin{equation}\label{eq:birmanexactseq}
    1\to \pi_1(\Sigma_g,p) \to \text{Mod}(\Sigma_g,p) \to \text{Mod}(\Sigma_g) \to 1
\end{equation}
where elements of $\pi_1(\Sigma_g,p)$ give mapping class elements that are ``push maps'' dragging along a given loop based at $p$. The moduli space $M_g$ has the concrete description
\begin{equation}\label{eq:moduli-concrete}
    M_g = \{(A_1,\ldots,A_g,B_1,\ldots,B_g) \in SU(2)^{2g} : [A_1,B_1]\cdots [A_g,B_g]= -1\}/SU(2)
\end{equation}
where the $SU(2)$-action conjugates a tuple $(A_i,B_i)$ to $(gA_ig^{-1},gB_ig^{-1})$ for a given $g\in SU(2)$. The identification \eqref{eq:moduli-concrete} follows from the description of $M_g$ in terms of flat $SO(3)$-connections. One lifts the $SO(3)$-bundle over $\Sigma_g\setminus \{p\}$ to a trivialized $SU(2)$-bundle, and $A_i,B_i$ are defined to be the holonomies of an $SU(2)$ lift of a given flat $SO(3)$-connection to this trivialization, where the holonomies are taken along a standard generating basis for the fundamental group. In actuality, to make these holonomies explicit, one should choose a basepoint in $\Sigma_g\setminus \{p\}$ to base all loops, but we will suppress this choice in what follows.

The Lagrangians $L_i= L(H_i^\circ,w_i)\subset M_g$ for $i\in \{1,2\}$ may be described in similar terms. Consider first $L_2$, which is the space of flat $SO(3)$-connections in $M_g$ that extend to a flat connection on the manifold-bundle data $(H_2^\circ,w_2)$. View $H^\circ_2$ as $[-1,1]\times T^2$ with $(g-1)$-many $1$-handles attached to $\{1\}\times T^2$. The arc $w_2=w_2''$ may be viewed as $[-1,1]\times \{q\}$ where $q$ is away from the $1$-handle attachments, and $p=(1,q)$. Extend the trivialized $SU(2)$-bundle over $\Sigma_g\setminus \{p\}$ from above to $H_2^\circ\setminus w_2$. Then the flat connections of interest are those that extend to this trivialized $SU(2)$-bundle and which have holonomy $-1$ around meridional loops of $w_2$. Then we may identify
    \begin{equation}\label{eq:lagrangiandescription}
        L_2 = \{(A_i,B_i) \in SU(2)^{2g}\; : \; B_i=1 \; (1\leq i \leq g-1), \; [A_g,B_g]= -1\}/SU(2).
    \end{equation}
Indeed, the holonomy $B_i$ for $1\leq i\leq g-1$ is along a loop in $\Sigma_g\setminus \{p\}$ which we may arrange to bound the cocore disk of the $i^{th}$ 1-handle of $H_2^\circ$, yielding $B_i=1$ in these cases. The elements $A_g,B_g$ are the holonomies around the two circle factors of $\{1\} \times T^2$.

The case of $L_1$ is similar. In this case recall that $w_1$ consists of two components $w_1'$ and $w_1''$. The only change in the above discussion involves $B_1$, say, which records the holonomy around a meridian of the loop $w_1'$. The condition that the trivialized $SU(2)$-bundle over $H_1^\circ\setminus w_1$ does not extend over $w_1'$ translates into the condition $B_1=-1$. We obtain
    \begin{equation}\label{eq:lagrangiandescription2}
        L_1 = \{(A_i,B_i) \in SU(2)^{2g}\; : \; B_1=-1, \; B_i=1 \; (2\leq i \leq g-1), \; [A_g,B_g]= -1\}/SU(2).
    \end{equation}
We remark that for $i\in \{1,2\}$ there are diffeomorphisms
    \begin{equation}\label{eq:lagrangiandescriptiondiff}
        L_i = L(H_i^\circ,w_i) \cong SU(2)^{g-1}.
    \end{equation}
    Indeed, the diffeomorphisms are determined by taking a tuple $(A_i,B_i)$ in either of the cases \eqref{eq:lagrangiandescription} and \eqref{eq:lagrangiandescription2}, conjugating the tuple so that we have
    \[
        A_g = \left[\begin{array}{cc} i &  0 \\ 0 & -i \end{array} \right], \qquad B_g = \left[\begin{array}{cc} 0 &  1 \\ -1 & 0 \end{array} \right],
    \]
    and then forgetting $B_1,\ldots,B_g,A_g$ to obtain $(A_1,\ldots,A_{g-1})\in SU(2)^{g-1}$.

A diffeomorphism $\phi$ as above induces a symplectomorphism $f_\phi:M_g\to M_g$. In fact, $f_\phi$ depends only on $[\phi]\in \text{Mod}(\Sigma_g,p)$. Considering $f_\phi$ up to symplectic isotopy leads to a homomorphism $\text{Mod}(\Sigma_g,p) \to \pi_0 \text{Symp}(M_g)$. As explained by Smith \cite[\S 2.1]{smith}, this construction factors through a homomorphism $\widehat{\rho}:\widehat{\Gamma}_g\to \pi_0 \text{Symp}(M_g)$ where the group $\widehat{\Gamma}_g$ is a quotient of $\text{Mod}(\Sigma_g,p)$ and fits into the exact sequence \eqref{eq:theexactsequence}. This follows from the fact that the action on $M_g$ induced by the push map associated to a class in $\pi_1(\Sigma_g,p)$ depends only on the associated class in $H_1(\Sigma_g;\Z/2)$. Identifying $H_1(\Sigma_g;\Z/2)$ with $H^1(\Sigma_g;\Z/2)$ by Poincar\'{e} duality, the action is
\begin{equation}\label{eq:push}
    (a_i,b_i)\cdot (A_i,B_i) =  ((-1)^{a_i}A_i,(-1)^{b_i}B_i)
\end{equation}
for $(a_i,b_i)\in (\Z/2)^{2g} \cong H^1(\Sigma_g;\Z/2)$, where this last identification is induced by the same generating set of loops as is used in describing $M_g$.

\vspace{0.15cm}

\begin{prop}\label{prop:nonvanishingimpliesfaithful}
Let $\phi:(\Sigma_g,p)\to (\Sigma_g,p)$ be a diffeomorphism. If $I_\ast(Y_\phi,w_\phi)\neq 0$, then there exists a pair of simply-connected Lagrangians $L$ and $L'$ embedded in $M_g$ such that 
    \begin{equation*}\label{eq:lagrangianfloercomparison}
        {\rm HF}(L,L')\not\cong {\rm HF}(L,\widehat{\rho}(\phi)(L')).
    \end{equation*}
    In particular, $\widehat{\rho}(\phi)\neq 0 = [{\rm{id}}] \in \pi_0 {\rm Symp}(M_g)$. 
\end{prop}

\begin{proof}
    Writing $L_i=L(H^\circ_i,w_i)$ as above, Theorem \ref{thm:atiyahfloer} and the assumption $I_\ast(Y_\phi,w_\phi)\neq 0$ gives
    \begin{equation}\label{eq:lagrangianfloercomparison1}
        I_\ast(Y_\phi,w_\phi)_{\Sigma_g\cup T} \cong \text{HF}_\ast(L_1,\widehat{\rho}(\phi)(L_2))\neq 0.
    \end{equation}
    Here we conflate $\widehat{\rho}(\phi)$ with the representing symplectomorphism $f_\phi:M_g\to M_g$ induced by $\phi$. On the other hand, as $L_1$ and $L_2$ are disjoint, see \eqref{eq:lagrangiandescription}--\eqref{eq:lagrangiandescription2}, we have
    \begin{equation}\label{eq:lagrangianfloercomparison2}
       \text{HF}_\ast(L_1,L_2) = 0.
    \end{equation}
    One may also see this vanishing via Theorem \ref{thm:atiyahfloer}, which identifies the Lagrangian Floer homology group \eqref{eq:lagrangianfloercomparison2} with a subgroup of $I_\ast(Y_{\text{id}},w_{\text{id}})$. The latter group vanishes by Proposition \ref{prop:nonvanishinginstanton}, upon observing that $Y_\text{id} = \#^{g-1} S^1\times S^2 \# T^3$ and $w_{\text{id}}$ has odd pairing with one of the $2$-spheres in this decomposition, thanks to the core $w_1'\subset w_1$. In any case, as $L_1$ and $L_2$ are simply-connected by \eqref{eq:lagrangiandescriptiondiff}, the first claim of the proposition follows.   

    For the claim regarding $\widehat{\rho}$, suppose $\widehat{\rho}(\phi)=0$, so that $f_\phi$ is symplectic isotopic to the identity map on $M_g$. As $M_g$ is simply connected \cite[Cor. 2]{newstead-topological}, this implies that $f_\phi$ is Hamiltonian isotopic to the identity. Lagrangian Floer homology is invariant under Hamiltonian isotopies of either Lagrangian, and so \eqref{eq:lagrangianfloercomparison1} and \eqref{eq:lagrangianfloercomparison2} are isomorphic, a contradiction.
\end{proof}

\section{Proofs of theorems} \label{sec:proof}

We now prove Theorems \ref{thm:main} and \ref{thm:main-2}. Let $G$ be the subset of $\widehat{\Gamma}_g$ consisting of those $\phi$ such that 
\[
    \text{HF}(L,\widehat{\rho}(\phi)(L')) \cong \text{HF}(L,L')
\]
as groups, for all simply-connected Lagrangians $L$ and $L'$ in $M_g$.
The statement of Theorem \ref{thm:main-2} is equivalent, in the cases $g\geq 3$, to the subset $G$ consisting only of the identity class. Analogous to \cite[Prop. 2.2]{clarkson}, we claim that $G$ is a normal subgroup of $\widehat{\Gamma}_g$. To see this, first note that the identity element is in $G$, and if $\phi_1,\phi_2\in G$, then for any Lagrangians $L,L'$ as above we have
\begin{align*}
    \text{HF}(L,\widehat{\rho}(\phi_1\phi_2^{-1})(L')) &= \text{HF}(L,\widehat{\rho}(\phi_1)\left(\widehat{\rho}(\phi_2)^{-1}(L')\right)) \\[1mm]
    &\cong \text{HF}(L,\widehat{\rho}(\phi_2)^{-1}(L')) = \text{HF}(\widehat{\rho}(\phi_2)(L),(L'))\\[1mm]
    & \cong \text{HF}(L,L')
\end{align*}
where the first line to second uses $\phi_1\in G$ and the second to third uses $\phi_2\in G$. Thus $G$ is a subgroup. To see that it is normal, let $\phi\in G$ and $\psi\in \widehat{\Gamma}_g$. Then
\begin{align*}
    \text{HF}(L,\widehat{\rho}(\psi^{-1}\phi\psi)(L')) & = \text{HF}(\widehat{\rho}(\psi)(L),\widehat{\rho}(\phi)\left(\widehat{\rho}(\psi)(L')\right)) \\[1mm]
    &  \cong 
    \text{HF}(\widehat{\rho}(\psi)(L),\widehat{\rho}(\psi)(L')) = \text{HF}(L,L')
\end{align*}
where the first line to second uses $\phi\in G$. Thus $\psi^{-1}\phi\psi\in G$, and $G$ is normal.

\vspace{0.15cm}

\begin{lemma}\label{lemma:push}
    Let $\phi:(\Sigma_g,p)\to (\Sigma_g,p)$ be a push map representing a nontrivial class in the image of the homomorphism $H^1(\Sigma_g;\Z/2)\to \widehat{\Gamma}_g$. Then $[\phi]\not\in G$.
\end{lemma}

\begin{proof}
Let $\phi$ be the push map which pushes $p\in \Sigma_g$ through the loop corresponding to $A_1$ in the description \eqref{eq:moduli-concrete}. By Poincar\'{e} duality this corresponds to the class $(a_i,b_i)\in(\Z/2)^{2g}\cong H^1(\Sigma_g;\Z/2)$ with $b_1=1$ and all other entries zero. The action of this element on $M_g$ is given by \eqref{eq:push}, induced by changing the sign of $B_1$. In particular, we have $f_\phi(L_1)=L_2$. Invoking Theorem \ref{thm:atiyahfloer}, we have
\begin{equation}\label{eq:pushdiffeo2}
    \text{HF}_\ast(L_2,\widehat{\rho}(\phi)(L_1)) = \text{HF}_\ast(L_2,L_2) \cong I_\ast(Y_{\text{id}},u)_{\Sigma_g\cup T}
\end{equation}
where $u$ is the union of the two arcs $w_1''$ and $w_2''$. In particular, unlike $w_{\text{id}}$, the $1$-manifold $u$ does not contain the circle $w_1''$. Observing $Y_{\text{id}} = \#^{g-1} S^1\times S^2 \# T^3$ and that $u$ is a circle factor of $T^3$ in this decomposition, we have $I_\ast(Y_{\text{id}},u)\neq 0$ by Proposition \ref{prop:nonvanishinginstanton}, and thus \eqref{eq:pushdiffeo2} is non-zero. That $[\phi]\not\in G$ now follows from Proposition \ref{prop:nonvanishingimpliesfaithful}. The proof is completed by noting that any other $\phi$ as in the statement is related to the above one by some diffeomorphism of $\Sigma_g$.
\end{proof}

\vspace{0.15cm}

Here is another viewpoint of the above argument. Let $\phi$ be the push map as in the proof. Then $\phi$ is isotopic to the identity (forgetting the basepoint) by an isotopy which undoes the pushing of the point around the loop corresponding to $A_1$. This induces a diffeomorphism
\begin{equation}\label{eq:pushdiffeo}
  Y_\phi\cong Y_{\text{id}} = \#^{g-1} S^1\times S^2 \# T^3.
\end{equation}
However, under this diffeomorphism, the $1$-manifold $w_\phi$ does not go to the isotopy class of $w_{\text{id}}$. Indeed, While $w_{\text{id}}$ simply passes through $p\in \Sigma_g\subset Y_{\text{id}}$ going from one compression body to the other, $w_\phi$ as viewed in $Y_{\text{id}}$ is isotopic to a curve which is $w_\text{id}$ away from $\Sigma_{g}$ and which at $\Sigma_{g}$ stops at $p$ to go around the loop corresponding to $A_1$. As this latter loop is homologically the same as $w_1'$ in $H_1^\circ$, we see that under \eqref{eq:pushdiffeo}, $[w_\phi]$ goes to $[u]\in H_1(Y_{\text{id}};\Z/2)$, the circle factor of $T^3$.

\vspace{0.15cm}

\begin{lemma}\label{lemma:reduction}
    Assume $g\geq 3$. To establish $G=\{1\}$, it suffices to show that $[\phi]\not\in G$ for all $\phi:(\Sigma_g,p)\to (\Sigma_g,p)$ whose image in ${\rm{Mod}}(\Sigma_g)$ is Torelli and pseudo-Anosov.
\end{lemma}

\begin{proof}
Let $K\subset \text{Mod}(\Sigma_g)$ be the image of $G$ under $\widehat{\Gamma}_g\to \text{Mod}(\Sigma_g)$. Note that $K$ is a normal subgroup of the mapping class group. By \cite[Prop. 2.3]{clarkson}, assuming $g\geq 3$, any normal subgroup of $\text{Mod}(\Sigma_g)$ which contains no pseudo-Ansosov elements of the Torelli group is trivial.

From the exact sequence \eqref{eq:theexactsequence}, we obtain that $G$ is in the image of $H^1(\Sigma_g;\Z/2)\to \widehat{\Gamma}_g$. However, Lemma \ref{lemma:push} shows that the intersection of $G$ with this image is the trivial group.
\end{proof}

\vspace{0.15cm}

A variation on the gluing construction \eqref{eq:ouradmissibledecomposition} yields the $3$-manifold
\begin{equation*}\label{eq:heegaardsplitting}
    Y_\phi' := H_1\cup_{\Sigma_g} H_2
\end{equation*}
which is formed by gluing $\partial H_1\to \partial H_2$ using $\phi$. As $H_1$ and $H_2$ are handlebodies, this construction endows $Y_\phi'$ with a preferred Heegaard splitting. Alternatively, $Y_\phi'$ is the result of cutting $Y_\phi$ along $T$ and on each side filling in the resulting $T$-component of $\partial H_i^\circ$ to obtain $H_i$.

Suppose two surfaces $R$ and $S$ in a $3$-manifold intersect transversely in a circle which bounds a disk $D\subset R$. Take a small regular neighborhood of $D$ identified with $D\times [-1,1]$ for which $S\cap D\times [-1,1]= \partial D\times [-1,1]$. Then {\emph{surgery of $S$ along $D$}} is the operation of replacing $S$ with the surface which is the union of $S\setminus D\times [-1,1]$ and $D\times \{-1,1\}$.

\begin{figure}[t]
\centering
\begin{tikzpicture}
    \node at (0,0) {\includegraphics[scale=1.7]{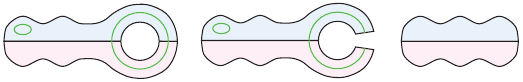}};
    \node at (0.5,1.65) {$Y_\phi^\circ$};
    \node at (-5,1.65) {$Y_\phi$};
    \node at (5.75,1.65) {$Y_\phi'$};
    \draw[gray] (-5.2,-1.2) -- (-5.2,0);
    \node at (-5.1,-1.5) {$\Sigma_g$};
    \draw[gray] (-7.15,0.9) -- (-6.85,0.47);
    \node at (-7.35,1.2) {$w_1'$};
    \draw[gray] (-7.35,-1.2) -- (-6.85,-0.4);
    \node at (-7.35,-1.5) {$H_2^\circ$};
    \draw[gray] (-2.5,0.95) -- (-2.8,0.45);
    \node at (-2.45,1.2) {$w_1''$};
    \draw[gray] (-6.2,0.95) -- (-6.2,0.2);
    \node at (-6.2,1.2) {$H_1^\circ$};
    \draw[gray] (-2.3,-1.3) -- (-2.55,-0.02);
    \node at (-2.25,-1.52) {$T$};
    \draw[gray] (-3.45,-1.25) -- (-3.45,-0.82);
    \node at (-3.45,-1.5) {$w_2''$};
\end{tikzpicture}
\caption{\small Schematic diagrams of $Y_\phi$, $Y_{\phi}^\circ$ and $Y_\phi'$. First, $Y_{\phi}$ is obtained by gluing $H_1^\circ$ to $H_2^\circ$ along $\Sigma_g$ using $\phi$ and along $T$ using the identity. $Y_\phi^\circ$ is obtained from $Y_\phi$ by cutting along $T$. $Y_\phi'$ is obtained from $Y_\phi^\circ$ by filling in the resulting $T$-components; it is the Heegaard splitting formed by the two handlebodies $H_1$ and $H_2$ glued along $\Sigma_g$ via $\phi$. The curves $w_1=w_1'\cup w_1''\subset H_1^\circ$ and $w_2=w_2''\subset H_2^\circ$ are also indicated.}
\label{fig:gluing}
\end{figure}
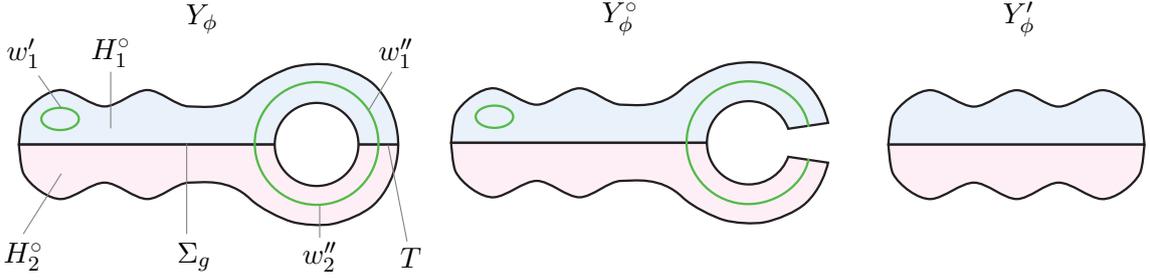

\vspace{0.15cm}

\begin{lemma}\label{lemma:irreducible}
    Suppose that $\phi$ descends to a mapping class in the Torelli group, and that the Heegaard splitting of $Y_\phi'$ is irreducible. Then $Y_\phi$ is irreducible.
\end{lemma}

\begin{proof}
    Suppose $Y_\phi$ is reducible. We first argue that a reducing sphere can be chosen so that it does not intersect $T\subset Y_\phi$. Let $S\subset Y_\phi$ be a reducing sphere for which the number of components of $S \cap T $ is minimal among all reducing spheres which are transverse to $T$. Suppose there is a component $\alpha$ of $S\cap T$ that bounds a disk $D\subset T$. Do surgery of $S$ along $D$ to obtain spheres $S_1$ and $S_2$ each of whose intersection with $T$ has fewer components than did $S$. Suppose each $S_i$ bounds a ball $B_i\subset Y_\phi$. If $B_1\cap B_2=\emptyset$, then gluing them back along $D$ shows that $S$ bounds a ball, a contradiction. Otherwise, without loss of generality, the interior of $B_1$ contains $B_2$. In this case, we can isotope $S$ to remove $\alpha$ from $S\cap T$ by pushing $S\cap S_2$ through $D$, contradicting minimality as well. Thus $S\cap T$ does not contain any component which bounds a disk in $T$.

    It follows that $S\cap T$ consists of a collection of parallel essential simple closed loops in the torus $T$. Let $N(T)\subset Y_{\phi}$ be an open regular neighborhood of $T$ and consider 
    \begin{equation}\label{eq:heegaardsplittingcompression}
        Y_\phi^\circ:= Y_\phi \setminus N(T) \cong H_1^\circ \cup_{\Sigma_g} H_2^\circ
    \end{equation}
    where the gluing along $\Sigma_g$ is by $\phi$. See Figure \ref{fig:gluing}. Then $S':= S\setminus N(T)$ is a planar surface in $Y_\phi^\circ$. Note that $H_1(T;\Z)\to H_1(Y_\phi^\circ;\Z)$ is injective by the construction of $Y_\phi^\circ$ and the assumption that $\phi$ acts trivially on $H_\ast(\Sigma_g;\Z)$. Therefore $S'$ has no disk components. Observe that $S$ is obtained from $S'$ by gluing along its circle boundary components, so $\chi(S)=\chi(S')$; since $S'$ is planar with no disks, $\chi(S')<0$, implying $\chi(S)<0$, a contradiction.

    Thus a reducing sphere $S$ can be chosen such that $S\cap T=\emptyset$, and we may view $S\subset Y^\circ_\phi$. The decomposition \eqref{eq:heegaardsplittingcompression} is a Heegaard splitting of a compact $3$-manifold into two compression bodies in the sense of Casson and Gordon \cite{cg-heegaard}. Their Lemma 1.1, which is a variation on a theorem of Haken, guarantees the existence of a reducing sphere $S$ that intersects $\Sigma_g$ transversely in one circle. But $S$ may also be viewed inside the closed $3$-manifold $Y_\phi'$ with its Heegaard splitting, which by assumption is irreducible. Thus $S\cap \Sigma_g$ must bound a disk on $\Sigma_g$. Do surgery of $S\subset Y^\circ_\phi$ along this disk to obtain two spheres $S_i^\circ\subset H_i^\circ$ where $i\in \{1,2\}$. Since compression bodies are irreducible, each of $S_i^\circ$ bounds a ball in $H_i^\circ$. Gluing these balls back along the surgery disk shows that $S$ bounds a ball, a contradiction. Thus no such reducing sphere exists, and $Y_\phi$ is irreducible.
\end{proof}

\vspace{0.15cm}

\begin{proof}[Proof of Theorem \ref{thm:main-2}]
    Let $\phi:(\Sigma_g,p)\to (\Sigma_g,p)$ be any diffeomorphism that descends to a pseudo-Anosov mapping class in the Torelli group. By Lemma \ref{lemma:reduction} and the assumption $g\geq 3$, to prove the theorem it suffices to show that $[\phi]\not\in G$. 

    We explain how the arguments of \cite{clarkson} adapt to yield a diffeomorphism $\eta$ closely related to $\phi$ for which the Heegaard splitting of $Y_\eta'$ is irreducible. First recall that the curve complex $C(\Sigma_g)$ is the simplicial complex whose vertices are isotopy classes of essential simple closed curves on $\Sigma_g$ and $n$-simplices corresponding to collections of $(n+1)$-many disjoint essential closed curves. A distance function on vertices of $C(\Sigma_g)$ is defined by assigning length $1$ to the edges in $C(\Sigma_g)$.

    The genus $g$ handlebody $H$ determines a set of vertices $\mathcal{V}_H \subset C(\Sigma_g)$ given by the essential simple closed curves that bound disks in $H$. Let $\phi_0$ be a diffeomorphism of $\Sigma_g$. As observed by Hempel, $d(\mathcal{V}_H,\phi_0(\mathcal{V}_H))\geq 1$ if and only if the Heegaard splitting of $Y_{\phi_0}$ is irreducible. Furthermore, by \cite{hempel,abrams-schleimer}, if $\phi_0:\Sigma_g\to \Sigma_g$ is pseudo-Anosov with stable and unstable laminations not in the closure of $\mathcal{V}_H$, all viewed inside the space of projective measured laminations of $\Sigma_g$, then
    \begin{equation}
        \lim_{n\to \infty} d(\mathcal{V}_H, \phi_0^n(\mathcal{V}_H)) = \infty.\label{eq:hempeldistance}
    \end{equation}
    While \eqref{eq:hempeldistance} might not hold for the given $\phi$, the argument of \cite[Lemma 4.7]{clarkson} shows that there exists another pseudo-Anosov $\psi:\Sigma_g\to \Sigma_g$ for which the stable and unstable laminations of $\phi_0:=\psi \phi\psi^{-1}$ are not in the closure of $\mathcal{V}_H$. Then there exists an integer $N>0$ such that
    \[
        \eta := \psi \phi^N\psi^{-1}
    \]
    has $d(\mathcal{V}_H, \eta(\mathcal{V}_H))\geq 1$, and in particular the Heegaard splitting of $Y'_{\eta}$ is irreducible.
    
    Isotope $\psi$ to satisfy $\psi(p)=p$, and form the pair $(Y_\eta,w_\eta)$ using $\eta$. Given that the Heegaard splitting of $Y_\eta'$ is irreducible, Lemma \ref{lemma:irreducible} implies $Y_\eta$ is irreducible, and Proposition \ref{prop:nonvanishinginstanton} yields
    \begin{equation*}\label{eq:nonvanishinginproof}
        I_\ast(Y_\eta,w_\eta)\neq 0.
    \end{equation*}
    From this non-vanishing and Proposition \ref{prop:nonvanishingimpliesfaithful} we obtain 
    $[\eta]\not\in G$. As $\eta = \psi\phi^N\psi^{-1}$ and $G$ is normal, we obtain $[\phi]^N\not\in G$, and hence $[\phi]\not \in G$, completing the proof.
\end{proof}

\vspace{0.15cm}

\begin{proof}[Proof of Theorem \ref{thm:main}]
For $g\geq 3$, Theorem \ref{thm:main} is implied by Theorem \ref{thm:main-2}, and so we may assume $g=2$. By \cite[Prop. 2.3]{clarkson}, in the case $g =2$, any normal subgroup of $\text{Mod}(\Sigma_g)$ which contains no pseudo-Ansosov elements of the Torelli group is either the trivial subgroup or is the order $2$ subgroup consisting of the hyperelliptic involution. Thus when $g=2$, the above arguments imply that $K$, the image of $G$ under $\widehat{\Gamma}_g\to \text{Mod}(\Sigma_g)$, is one of these two subgroups.

Let $G'$ be the kernel of $\widehat{\rho}$ and $K'$ its image in $\text{Mod}(\Sigma_g)$. Then $G'\subset G$ and $K'\subset K$. To show that $K'$ is trivial, we must show that $\widehat{\rho}(\phi) = [f_\phi]\neq 0$ where $f_\phi:M_2\to M_2$ is associated to a hyperelliptic involution $\phi$. Following for example \cite{thaddeus-intro}, there is an isomorphism 
\[
    H^3(M_2;\Z)\cong H_1(\Sigma_2;\Z)
\]
which intertwines the action of $(f_{\phi})_\ast$ on $H^3(M_2;\Z)$ with that of $\phi_\ast$ on $H_1(\Sigma_2;\Z)$. For $\phi$ a hyperelliptic involution, $\phi_\ast$ is negation, and thus $(f_\phi)_\ast$ is not the identity. In particular, $f_\phi$ is not even homotopic to the identity, and consequently $\widehat{\rho}(\phi)=[f_\phi]\neq 0$. Thus $K'$ is trivial.

From the exact sequence \eqref{eq:theexactsequence}, we obtain that $G'$ is contained in the image of $H^1(\Sigma_g;\Z/2)$ in $\smash{\widehat{\Gamma}_g}$. However, as $G$ intersects this image trivially, so too does $G'$. (In fact, Smith already showed in \cite[Lemma 2.9]{smith} that $\widehat{\rho}$ is injective on this image.) Thus $G'=\text{ker}(\widehat{\rho})$ is trivial.
\end{proof} 

\vspace{0.15cm}

\begin{remark}\label{rmk:genus2}
{\normalfont
    In the case $g=2$, the proof of Theorem \ref{thm:main-2} shows that the kernel of $\widehat{\rho}$ maps into an order two subgroup of $\text{Mod}(\Sigma_g)$. As Lemma \ref{lemma:push} implies that the intersection of $\text{ker}(\widehat{\rho})$ with the image of $H^1(\Sigma_g;\Z/2)$ is the trivial subgroup, we obtain that $\text{ker}(\widehat{\rho})$ itself is contained in an order two subgroup of $\widehat{\Gamma}_2$ generated by a hyperelliptic involution that fixes the basepoint $p\in \Sigma_g$.
}
\end{remark}

\section{Further questions}\label{sec:further}

Assume $g\geq 2$. As mentioned in the introduction, the question of whether $\widehat{\rho}$ is injective was already rasied in \cite[Remark 5.6]{dostoglou-salamon}, which also includes the following problem:

\vspace{0.15cm}

\begin{problem} \label{surj}
    Is $\widehat{\rho}:\widehat{\Gamma}_g \to \pi_0 {\rm{Symp}}(M_g)$ surjective? That is to say, is any symplectomorphism of $M_g$ symplectically isotopic to the symplectomorphism induced by an element of $\widehat{\Gamma}_g$?
\end{problem}

\vspace{0.15cm}

A related and perhaps more tractable question is the following. Let $H^\circ$ be the compression body introduced in Section \ref{sec:background} and $w$ be any properly embedded arc connecting the boundary components of $H^\circ$. Let $L=L(H^\circ,w)$ be the associated Lagrangian in $M_g$. Applying $\widehat{\rho}(\phi)$ to $L$ as $\phi$ ranges over $ \smash{\widehat{\Gamma}_g}$ gives a collection of Lagrangians in $M_g$. 
\vspace{0.15cm}

\begin{problem} \label{gen}
    Does the family of Lagrangains 
    $\{\widehat{\rho}(\phi)(L)\}_{\phi\in \widehat{\Gamma}_g}$
    split generate $D^\pi \mathscr F(M_g;0)$?
\end{problem}
\vspace{0.15cm}

\noindent In the case $g=2$, this problem was answered affirmatively in \cite[Lemma 4.15]{smith}. We refer the reader there for the precise definition of $D^\pi \mathscr F(M_g;0)$, which is the summand of the derived Fukaya category $D^\pi \mathscr F(M_g)$ of $M_g$ given by monotone Lagrangians whose potential classes are $0$. The Lagrangian $L$, and hence all the Lagrangians in the above collection, are orientable, spin and monotone with Maslov number $4$. In particular, they have trivial potential class and belong to the summand $\mathscr F(M_g;0)$ of the Fukaya category of $M_g$.

Smith's approach to prove Theorem \ref{thm:main} for $g=2$ is based on the notion of {\it Floer-theoretic entropy}. For a symplectomorphism $\psi:M_g \to M_g$, he defines the Floer-theoretic entropy of $\psi$ as 
\[
  h_{Floer}(\psi):= \limsup{\frac{1}{n}\log {\rm rk}(HF(\psi^n))}
\]
where $HF(\psi^n)$ denotes the fixed Floer homology of the symplectomorphism $\psi$. In the case that $\psi=\widehat{\rho}(\phi)$ for $[\phi]\in \text{Mod}(\Sigma_g,p)$, this Floer homology group is isomorphic to the instanton Floer homology of $(M_\phi,w_\phi)$, the mapping torus of $\phi$ \cite{DS:AF-cyl}. The following is implicit in \cite{smith}:
\vspace{0.15cm}

\begin{problem} \label{surj}
    Let $\phi \in \widehat{\Gamma}_g$ with $g\geq 2$ have a pseudo-Anosov component. Does $\widehat{\rho}(\phi)$ have strictly positive Floer theoretic entropy?
\end{problem}

\vspace{0.15cm}

\noindent Smith shows that a positive answer to this question implies that $\widehat{\rho}$ is faithful. Furthermore, he shows the answer to this question is positive when $g=2$. Our proof of Theorem \ref{thm:main} does not immediately say anything about the entropy of $\widehat{\rho}(\phi)$. 

We may consider similar problems of moduli spaces for other principal $G$-bundles over $\Sigma_g$. For example, denote by $M_g^{n,d}$ the moduli space of flat connections on some $U(n)$-bundle $P\to \Sigma_g$ with fixed determinant connection, where $c_1(P)[\Sigma_g]=d$. This moduli space has the description
\begin{equation*}\label{eq:higherrankreps}
    M_g^{n,d} = \{(A_1,\ldots,A_g,B_1,\ldots,B_g) \in SU(n)^{2g} : [A_1,B_1]\cdots [A_g,B_g]= \zeta^d I\}/SU(n)
\end{equation*}
where $\zeta=e^{2\pi\sqrt{-1} /n}$ and the action is by conjugation. The moduli space $M_g$ studied in this paper is $M_g^{2,1}$. In general, when $n$ and $d$ are coprime, $M_{g}^{n,d}$ is a smooth symplectic manifold of dimension $M_g^{n,d}=2(n^2-1)(g-1)$, and diffeomorphisms $\Sigma_g\to\Sigma_g$ preserving a basepoint naturally induce symplectomorphisms as before. The diffeomorphisms induced by push maps from $\pi_1(\Sigma_g,p)$ in \eqref{eq:birmanexactseq} factor through an action of $(a_i,b_i)\in H^1(\Sigma_g;\Z/n)=(\Z/n)^{2g}$ that sends $(A_i,B_i)$ to $(\zeta^{a_i}A_i,\zeta^{b_i}B_i)$. Therefore there is an extension generalizing \eqref{eq:theexactsequence},
\begin{equation*}\label{eq:theexactsequencerankn}
    1\to H^1(\Sigma_g;\Z/n) \to \widehat{\Gamma}^{n,d}_g \to \text{Mod}(\Sigma_g) \to 1,
\end{equation*}
and an induced homomorphism $\widehat{\rho}$ from $\widehat{\Gamma}^{n,d}_g$ to the symplectic mapping class group of $M_g^{n,d}$. 

\vspace{0.15cm}

\begin{conjecture}\label{conj:rankn}
    For $(n,d)$ coprime, $\widehat{\rho}:\widehat{\Gamma}^{n,d}_g \to \pi_0 {\rm{Symp}}(M^{n,d}_g)$ is injective.
\end{conjecture}

\vspace{0.15cm}

\noindent The authors expect that the strategy of this paper used for the case $(n,d)=(2,1)$ may be adapted to prove the more general Conjecture \ref{conj:rankn}. The relevant $U(n)$ instanton homology groups are available from work of Kronheimer and Mrowka \cite{yaft}. The remaining essential ingredients needed for the strategy are an analogue of Proposition \ref{prop:nonvanishinginstanton}, which already follows in the case $n=3$ from \cite{dis}, and a higher rank analogue of the Atiyah--Floer Conjecture result of Theorem \ref{thm:atiyahfloer}, which the authors expect is within reach by the methods already utilized in the $U(2)$ case.

The above problems may also be formulated in the cases in which $(n,d)$ are not coprime. For simplicity we focus our attention on the particular case of the $SU(2)$ moduli space
\[
    M_g^{2,0} = \text{Hom}(\pi_1(\Sigma_g),SU(2))/SU(2).
\]
In this case, the mapping class group of $\Sigma_g$ itself acts on $M_g^{2,0}$. It would be interesting to see to what extent the analogue of $\widehat{\rho}$ in this setting is injective. A problem along these lines for $M_g^{2,0}$ was proposed by Robert Lipshitz for the upcoming K3 problem list, an update to Kirby's problem list \cite{Kirby}. As $M_g^{2,0}$ is not smooth, to formulate this problem one must decide on a suitable definition of the symplectic mapping class group. For example, the natural symplectic form is still defined on the dense smooth stratum of irreducibles in $M_g^{2,0}$ and a diffeomorphism may be considered symplectic if it preserves the symplectic form on this stratum. We remark that for $g=2$, the work of Ruberman \cite{ruberman-mutation} implies that the hyperelliptic involution acts trivially on $M_2^{2,0}$.

Alternatively, we may view the action of the mapping class group on the $SU(2)$ character variety $M_g^{2,0}$ in terms of {\it extended moduli spaces} introduced in \cite{Jeffrey:extended,Huebschmann:extended}. The representation theoretic definition of the extended moduli space corresponding to $M_g^{2,0}$ is given as 
\[
  N_g:=\{(\rho,\zeta) \; : \; \rho: \pi_1(\Sigma_g\setminus p)\to SU(2),\, \zeta\in \mathfrak{su}(2),\; \rho \text{ is a homomorphism},\, \rho(\mu)=\exp(2\pi \zeta)\},
\]
where $\mu$ is a small loop around $p$. Let $N_g^{<r}$ denote the subspace of $N_g$ given by the pairs $(\rho,\zeta)$ with $|\zeta|<r$, where we use the convention that $|{\rm diag}(i,-i)|=1/4$. For $r\leq 1$, the space $N_g^{<r}$ is a smooth manifold. Further, there is a closed 2-form on $N_g^{<r}$ which is a symplectic form for $r\leq 1/2$. Simultaneously conjugating $\rho$ and $\zeta$ defines an action of $SU(2)$ on $N_g$ which factors through $SO(3)$. For $r=1/2$, this action on $N_g^{<r}$ is Hamiltonian and the moment map $\Phi: N_g^{<r} \to \mathfrak{su}(2)$ is given by projecting $(\rho,\zeta)$ to $\zeta$. From this description, it is clear that the symplectic quotient
\[
    N_g^{<r} /\!/SO(3):=\Phi^{-1}(0)/SO(3)
\]
for any $r\leq 1/2$ is homeomorphic to $M_g^{2,0}$ and this identification is a symplectomorphism away from the singular locus of $M_g^{2,0}$. One can then reformulate questions related to the singular space $M_g^{2,0}$ to equivariant questions involving $N_g^{<r}$. For instance, a slight variation of $N_g^{<r}$ was used to formulate the Atiyah--Floer conjecture for $SU(2)$-bundles \cite{MW:ext-Floer,DF:proc,Caz:eq-Lag}. 

Analogous to the case of the odd character variety, any $[\phi]\in \text{Mod}(\Sigma_g,p)$ determines a homeomorphism of $N_g$, which induces a symplectomorphism of $N_g^{<r}$ for $r\leq 1/2$. Any such symplectomorphism is equivariant with respect to the $SO(3)$ action on $N_g^{<r}$. As one possible formulation of the above faithfulness problem, it would be interesting to identify the kernel of the above homomorphism from $\text{Mod}(\Sigma_g,p)$ to the equivariant symplectic mapping class group of $N_g^{<r}$.

There is yet another family of character varieties that may be considered. For any odd positive integer $k$, let $\pi$ be a set of $k$ points in $\Sigma_g$, and define
\[
  M_{g,k}:=\{\rho:\pi_1(\Sigma_{g}\setminus \pi) \to SU(2) \; : \; \rho \text{ is a homomorphism},\; \text{Tr$\rho(\mu)=0$ }\}/SU(2)
\]
where $\mu$ runs over all small meridional loops around the elements of $\pi$. This is again a smooth Fano variety, which can be identified with a moduli space of parabolic bundles as well as a moduli space 
 of orbifold flat connections. Let $\text{Mod}(\Sigma_g,\pi)$ be the mapping class group of diffeomorphisms of $\Sigma_g$ that map $\pi$ to itself as a set. This gives rise to a homomorphism 
 \begin{equation}\label{eq:rhohat-orb}
\widehat{\rho}: \text{Mod}(\Sigma_g,\pi) \to \pi_0 \text{Symp}(M_{g,k}).
\end{equation}

\vspace{0.15cm}

\begin{conjecture}\label{conj:parb}
    Let $(g,k)$ be a pair of non-negative integers such that $k$
    is odd and $(g,k)\notin \{(0,1), (0,3), (1,1)\}$.
    Then the homomorphism $\widehat{\rho}: {\rm{Mod}}(\Sigma_g,\pi) \to \pi_0 {\rm{Symp}}(M_{g,k})$ is injective.
\end{conjecture}

\vspace{0.15cm}

\noindent The faithfullness of $\widehat{\rho}$ in \eqref{eq:rhohat-orb} for $(g,d)=(0,5)$ is established in \cite{S:Lec-Dehn}. The authors expect that an analogue of Theorem \ref{thm:atiyahfloer} may be proved, involving the moduli spaces $M_{g,k}$ and Kronheimer and Mrowka's singular instanton homology for links \cite{yaft}, and furthermore that the strategy of this paper may be adapted to address Conjecture \ref{conj:parb}.

\bibliographystyle{acm}
\bibliography{references}

\Addresses

\end{document}